\documentclass{article}
\usepackage{latexsym,amsmath,amssymb,amsthm,amsfonts, mathrsfs}
\usepackage[utf8]{inputenc}
\usepackage[T1]{fontenc}
\usepackage{color}

\newtheorem{theorem}{Theorem}

\newtheorem{proposition}[theorem]{Proposition}

\newtheorem{corollary}[theorem]{Corollary}
\newtheorem{lemma}[theorem]{Lemma}
\newtheorem{observation}[theorem]{Observation}
\newtheorem{definition}{Definition}
\newtheorem{conjecture}{Conjecture}
\newtheorem{question}{Question}

\newcommand{\prob}{\mathbb{P}}
\newcommand{\expect}{\mathbb{E}}
\newcommand{\trange}{t\in\{\lfloor n/2 \rfloor, \lceil n/2 \rceil\}}
\newcommand{\ceil}[1]{\lceil #1 \rceil}
\newcommand{\floor}[1]{\lfloor #1 \rfloor}

\begin{document}
\title{Random Uniform and Pure Random Simplicial Complexes}
\author{Klas Markstr\"om\thanks{klas.markstrom@umu.se} and Trevor Pinto\thanks{t.pinto@hotmail.co.uk}}
\maketitle

\begin{abstract}
	In this paper we introduce a method which allows us to study properties of the random uniform simplicial complex. That is, we assign equal probability to all simplicial complexes with a given number of vertices 
	and then consider properties of a complex under this measure. We are able to determine or present bounds for a number of topological and combinatorial properties. We also study the random pure simplicial complex 
	of dimension $d$, generated by letting any subset of size $d+1$ of a set of $n$ vertices be a facet with probability $p$ and considering the simplicial complex generated by these facets.  We compare the behaviour of 
	these models for suitable values of $d$ and $p$.
	
	Finally we use the equivalence between simplicial complexes and monotone boolean functions to study the behaviour of typical such functions. Specifically we prove that most monotone boolean functions are evasive, hence 
	proving that the well known Evasiveness conjecture is generically true for monotone boolean functions without symmetry assumptions. 
	
\end{abstract}

\section{Introduction}
The Erd\H{o}s-Renyi  random graph $G(n,p)$ has since its introduction been the most studied random model for a combinatorial object. Initially the focus was on the specific case $G(n,\frac{1}{2})$ which assigns equal probability to 
all graphs on $n$ vertices, and the aim was to understand what the typical properties of such graphs are.  More recently the study of random models for simplicial complexes has gathered a lot of attention. This was to a large extent triggered by the introduction of the Linial and Meshulam model in \cite{linialmeshulam}.  This model begins with a complete graph on $n$ vertices and includes facetts of size 3 independently with probability $p$. The model was 
generalised in \cite{meshulamwallach} by starting with a complete $k-1$-skeleton and adding $k$-sets independently with probability $p$.    The behaviour of these models has been studied in a long line of papers, with a particular emphasis on generalisations of connectivity, often in the form of homology  \cite{Ko10,ALLM,AL16,KP16,LP16,HKP} .   These models have also been generalised in the frame-wise uniform model of Brooke-Taylor and Testa \cite{BTT}, where faces of dimension $i$ have a probability $p_i$  of being included if all their subsets have already been included.  A number of properties of the BTT-model have been determined in \cite{MR3509567,MR3661651,MR3604492} for different choices of the values $p_i$.   This model also includes clique complexes of random graphs, a mode which had already been studied in a number of papers,  we refer to \cite{MR3290093} for a good survey. 

All these models are interesting from both combinatorial and topological perspectives,  and their analysis has required the development of new probabilistic tools as well as connections between probabilistic combinatorics and algebraic topology. However one aspect of the original Erd\H{o}s-Renyi   model which they have not covered is the study of typical simplicial complexes. None of the existing models can be made to assign equal probability to all simplicial complexes on a given number of vertices.  So the aim of this 
paper is to provide a way to study the uniform random simplicial complex. i.e.  the model where all complexes on $n$ vertices are given  equal probability.  This model is hard to work with directly but thanks to a structural result by Korshunov \cite{korshunovrussian}  we can restrict  our attention from the set  $\mathcal{M}(n)$  of all simplicial complexes to a smaller set which on one hand contains almost all elements of  $\mathcal{M}(n)$  and on the other hand is more easily analysed.   In this way we can prove that many properties hold with high probability in the uniform measure  $\mathcal{U}(n)$  on $\mathcal{M}(n)$.

We also study the random pure model $\mathcal{RP}(n,t,p)$ in which all the maximal faces are sets of size $t$.  Here a simplicial complex  $\Delta\sim \mathcal{RP}(n,t,p)$ is generated by letting subsets of $[n]$ of size $t$ to be 
facetts of $\Delta$ independently with probability $p$.  Next all non-empty subsets of the facetts are added.  We study the topological properties of this model and compare those with the uniform model. In particular, we find  
that $\mathcal{RP}(n,n/2,1/2)$ is a useful heuristic for $\mathcal{U}(n)$ when $n$ is even.

\subsection{The uniform model}
Our general strategy for handling the uniform random simplicial complex uses two stages. First  we use a structure theorem by Korshunov which demonstrates that there exists a family  $\mathcal{M}'(n)$ of simplicial complexes such that almost all complexes from $\mathcal{M}(n)$  belong to $\mathcal{M}'(n)$,  and $\mathcal{M}'(n)$ in turn has a partition into a collection of classes $\mathcal{M}(n,t, A, B)$, which we will soon 
describe.  Next, in order to analyse if a property $P$ holds in with high probability under the uniform distribution on  $\mathcal{M}(n)$  we prove that the property holds with high probability in $\mathcal{M}(n,t, A, B)$, for all 
relevant $t, A, B$.  In principle one can imagine properties $P$  which will be sensitive to the parameters  $t, A ,B$,  but as we shall see a wide range of  natural properties are not sensitive in thus way and can be treated by our methods.

Khorshunov's theorem in its original form is structural result for monotone boolean functions, which was proven in order to find the asymptotic number of such functions.  However, monotone Boolean functions and simplicial functions are equivalent object  and we here present his result in terms of simplicial complexes.   In order  state  Korshunov's  theorem, we require a bit of terminology from \cite{korshunovrussian}.  

Given  a family of $k$-sets $A$,  we decompose the family into \emph{bundles} by declaring two sets to be in the same bundle if their intersection has size $k-1$, and then making this relation transitive. 
\begin{definition}
	Let $n$ and $t$ be natural numbers with $t\leq n$. We say a pair of set families, $(A, B)$ is an \emph{admissible pair} if all of the following hold:
	\begin{enumerate}
		\item $A\subseteq [n]^{(t-1)}$ and $B\subseteq [n]^{(t+1)}$
		\item $\big||A|-2^n\big|\leq n2^{n/2+1}$.
		\item $\big||B|-2^n\big|\leq n2^{n/2+1}$.
		\item For all $a\in A$ and $b\in B$, we have that $b\nsubseteq a$.
		\item $A$ and $B$ both consist of one and two element bundles, and each have at most $16n^4$ two-element bundles.
	\end{enumerate}
\end{definition}
For admissible pairs $(A,B)$, we write $\mathcal{M}(n,t, A, B)$ for the collection of simplicial complexes on $n$ vertices such that $A$ is the set of facets of size $t-1$, $B$ is the set of facets of size $t+1$ and all the other facets have size $t$. Let also $\mathcal{M}(n,t)$ be the union of $\mathcal{M}(n,t, A, B)$  over all admissible $(A,B)$

The following theorem from \cite{korshunovrussian},  also stated in the more accessible \cite{korshunovsurvey} and expanded upon in \cite{korshunovshmulevich} by Korshunov and Shmulevich, gives us the structure of almost all simplicial complexes.
\begin{theorem}[Korshunov \cite{korshunovrussian}]\label{korshunov}
	{\ }

	If $n$ is even, then \[|\mathcal{M}(n)|=(1+o(1))|\mathcal{M}(n,n/2)|.\]

	If $n$ is odd, then \[|\mathcal{M}(n)|=(1+o(1))\Big(|\mathcal{M}(n,\lfloor n/2\rfloor)|+|\mathcal{M}(n,\lceil n/2 \rceil )|\Big).\]
\end{theorem}

We can now define the auxiliary random models which will help us in analysing    $\mathcal{U}(n)$.
\begin{definition} 
	For given  $t, A, B$  we define $\mathcal{U}(n,t, A, B)$  to be the uniform distribution on elements of $\mathcal{M}(n,t,A,B)$.
	Similarly we  define $\mathcal{U}(n,t)$ to be the uniform distribution on $\mathcal{M}(n,t)$ 
\end{definition} 
An element of $\mathcal{M}(n,t,A,B)$ is determined by the fixed elements from $A$ and $B$, and its facets of size $t$. Let $F$ be the collection of all $t$-sets which contain no set $a\in A$ and are not contained in any $b\in B$. We call such sets  \emph{free sets}. By the 
bounds on $A$ and $B$, we have that there are $(1+o(1))\binom{n}{t}$ free sets. Picking a uniformly random element, $\Delta$, of $\mathcal{M}(n,t,A,B)$ is equivalent to picking free sets to be faces of the complex 
independently with probability 1/2. This typically makes the distributions $\mathcal{U}(n,t, A, B)$ much easier to work with than working with $\mathcal{U}(n)$ directly, and if a property holds with high probability in $\mathcal{U}(n,t, A, B)$ for all admissible $(A,B)$ then it also holds with high probability  in $\mathcal{U}(n)$. 

As one may expect, given the form of Theorem \ref{korshunov}, some of the behaviour of $\mathcal{U}(n)$ depends strongly on the parity of $n$.  In order to get cleaner statements of our theorems we mainly work with even $n$ 
and then indicate how the corresponding result for odd $n$ can be derived.

\subsection{Overview of the paper}
The remainder of this paper is structured as follows. 

In Section \ref{sec:note} we define some of our  topological terminology  and other convenient notation.

Section \ref{sec:skeleton} we describe the skeleta of complexes from our models. We show that with high probability, $\Delta\sim \mathcal{U}(n)$ has a complete $(n/2-2)$-skeleton, when $n$ is even. We also find ranges of $p$ that guarantee, with high probability, complete $t'$-skeletons for $\Delta\sim \mathcal{RP}(n,t,p)$, when $t'$ is close to $t$.  These results both relate our models to e.g. the Linial-Meshulam model  and are also used in some of the later 
results.

In Section \ref{sec:hom}, we investigate the homology group in both models. The methods we use illustrate the similarity of working with $\mathcal{U}(n)$ and $\mathcal{RP}(n,t,p)$, particularly in the range of large $n$. We 
 provide lower bounds on the dimension of the homology group which  we conjecture are also best possible. 

In Section \ref{sec:Euler}, we discuss the Euler characteristic of $\mathcal{U}(n)$. We bound this in absolute value by $cn2^{n/2}$, for some constant $c$ and show that this  bound  is best possible up to some polynomial factor.

In Section \ref{sec:subcomp}, we  investigate which pure simplicial complexes appear as induced subcomplexes in $\Delta \sim \mathcal{U}(n)$ with high probability.

In Section \ref{sec:shell}, we look at the related notions of shellability and the $h$-vector. These properties  provide an instance where the random pure model behaves very differently to the uniform model.

Finally,  in Section \ref{evasivenesssec} we apply our work on $\mathcal{U}(n)$,  to  evasiveness of Boolean functions. We use a special case of our result on the homology of $\mathcal{U}(n)$ to prove that almost all 
monotone Boolean functions are evasive.

\section{Notation}\label{sec:note}

\subsection{Probability}
For general probabilistic facts we refer the reader to \cite{alonspencer}. 
We will make repeated use of the entropy function, often denoted by $H(c)$, and in order to avoid confusion with the terminology  from algebraic topology we instead write  $\alpha(c)=-c\log(c) -(1-c)\log(1-c)$. 

\begin{proposition}\label{constcbinom}
If $c$ is a constant and $0<c<1$, then, as $n$ tends to infinity,
$$\binom{n}{cn}= (1+o(1) ) \exp( \alpha(c)n  -1/2 \log(2 (1 - c) c n \pi)   )$$
\end{proposition}

\begin{lemma}\label{covariance}
	Suppose $X_1, \dots, X_m$ are indicator random variables. Then \[Var(X)\leq \expect(X) +\sum_{i\neq j} Cov(X_i,X_j).\]
\end{lemma}

\subsection{Simplicial complexes}

We introduce here the terminology that we use in the remainder of the paper.

A \emph{simplicial complex on $n$ vertices}, $\Delta$, is a collection of non-empty subsets of $[n]:=\{1,\dots n\}$ such that if $x\in \Delta$ and $\emptyset\neq y\subseteq x$ then $y\in \Delta$.

We refer to the elements of $\Delta$ as \emph{faces}. A \emph{facet} is a maximal face, that is, a face that is not contained in another face of the simplicial complex. The \emph{dimension} of a face $x$ is  $|x|-1$. In other words it is one less than the size of $x$ as a subset of $[n]$. Although in the literature dimension is used more frequently than size, for our purposes it is more convenient to work with the size of faces.

The \emph{link}  of a face  $x\in \Delta$ is the set $Lk(x)=\{ y | x\cup y \in \Delta,  x \cap y =\emptyset    \}$

We write  $[n]^{(k)}$ for the collection of size $k$ subsets of $[n]$. Given a simplex $\Delta$, its \emph{$k$-skeleton} is the complex formed by the collection of faces of $\Delta$ of size at most $k$. We say $\Delta$ has a 
complete $k$-skeleton if the $k$-skeleton is  $\cup_{j\leq t}[n]^{(j)}$.

\subsection{Algebraic topology}
We cannot hope to give an in depth review of algebraic topology, so the purpose here is instead to establish the terminology we will use.    We refer the reader to  the texbooks by Hatcher \cite{Ha02} and Edelsberunner  \cite{EdHa} for in depth discussion of algebraic topology, with a more classical point of view in the first and a more discrete and computational emphasis in the second. 

A $k$-\emph{chain}  is a formal linear combination $c=\sum_i c_i x_i$, where $x_i$ are $k$-simplices from $\Delta$, and the constant $c_i$ come from some ring $R$.     The set of all $k$-chains from $\Delta$ defines a group $C_k(\Delta,R)$ under addition.

The boundary operator $\partial_k$ maps a $k$-simplex  $x$ to the chain $\sum_i y_i $, where the $y_i$ are all the $(k-1)$-simplices contained in $x$. For a $k$-chain  the boundary is defined  as $\partial_k(c)= \sum_i c_i \partial_k(x_i)$. 
Under this definition the boundary operator is a linear operator $\partial_k C_k(\Delta,R) \rightarrow C_{k-1}(\Delta,R)$

The set $Z_k(\Delta,R)= \{ x\in C_k(\Delta,R) | \partial_k(x)=0  \}  $ is called the space of $k$-cycles,  and $B_k(\Delta,R)= \{ x\in C_{k}(\Delta,R) | \exists y\in C_{k+1}(\Delta,R),  x=\partial_{k+1}(y)   \}  $ is called the set of $k$-boundaries.

The $k$th homology group of $\Delta$ over $R$ is defined as $$H_k(\Delta,R)=Z_k(\Delta,R)/B_k(\Delta,R).$$
When $\Delta$ is clear from the context we will sometimes simply write $H_k$.
Sometimes it is also convenient to define the reduced homology groups, denoted $\bar{H}_k(\Delta,R)$.  Here $\bar{H}_k(\Delta,R)=H_k(\Delta,R)$ for $k>0$ and $H_0(\Delta,R)=\bar{H}_0(\Delta,R)\oplus \mathbb{Z}$.

The Euler Characteristic  of $\Delta$ can be defined in several equivalent ways:  
$$\chi(\Delta)=\sum_i (-1)^i f_i=  \sum_{r=0}^n \beta_r(\Delta)  $$  
where $f_i$ denotes the number of $i$-dimensional faces of $\Delta$, and the Betti number $\beta_r$ is given by the rank of $H_r(\Delta,R)$.

\section{The $k$-skeleton}\label{sec:skeleton}
Our first aim is to investigate the skeleta of complexes from $\mathcal{U}(n)$  and $\mathcal{RP}(n,t,p)$.  In particular we  will determine which of the skeleta are complete  in either of the two models. 

We start by considering the uniform random complex.  Korshunov's structure theorem \ref{korshunov} leds to a quick proof of the following.
\begin{theorem}\label{skeletonuniform}
	Suppose $\Delta\sim \mathcal{U}(n),$ where $n$ is even. Then with high probability, $\Delta$ has a complete $(n/2-2)$-skeleton. 
\end{theorem} 
\begin{proof}
	Given $\Delta \sim \mathcal{U}(n)$, we define $\Delta^c$ by the relation $x\notin \Delta \Rightarrow [n]\setminus x \in \Delta^c$. Note that $\Delta^c\sim \mathcal{U}(n)$ and hence, by Theorem \ref{korshunov}, it contains no 
	faces of size $n/2+2$ with high probability. Thus $\Delta$ has a complete $n/2-2$-skeleton with high probability.
\end{proof}
Note that we may use a similar method to show that when $n$ is odd, if $\Delta\sim \mathcal{U}(n,\ceil{n/2})$, then with high probability, $\Delta$ has a complete $(\ceil{n/2}-2)$-skeleton. Likewise, if $\Delta\sim \mathcal{U}(n,\floor{n/2}),$ then with high probability, $\Delta$ has a complete $(\floor{n/2}-2)$-skeleton.

For the random pure simplicial complex $\mathcal{RP}(n,t,p)$ we will consider two different ranges for $t$,  $t$ linear in $n$ and $t$ constant.   For both cases we may note that having a complete skeleton is a monotone property  under addition of new faces $t$-faces as well as under increases in $p$.   We will now prove that the property of having a  complete skeleton has a sharp threshold, and locate that threshold to leading order.  

For both this and several other results in this paper one could also develop a more refined stopping time version of the result. That is, we could consider a random process where at each time step a randomly chosen facet is added until the desired property is satisfied. This leads to a structured coupon collector problem \cite{FLM}  and e.g. the corresponding result for the existence of a perfect matching are among the classical results on random graphs.

\begin{theorem}\label{skeletonpure}
	Let  $t=cn$, where $0<c<1$ is a constant, $t'=t-k$ for some constant $k$, and  let $q(c,k)>0$  be a a large enough constant, Then, for  $\Delta\sim \mathcal{RP}(n,t,p)$:
	\begin{enumerate}
		\item For $k>1$, let $C=\frac{\alpha(c)k!}{(1-c)^k }$. 
		
		If $p=p_1:=\frac{1}{n^{k-1}}(C +  q (1-c)^{-k} k! \frac{\log(n)}{n})$,  then  $\Delta$ has a complete $t'$-skeleton with high probability. 
		
		If instead $p=p_2:=\frac{1}{n^{k-1}}(C -  q (1-c)^{-k} k! \frac{\log(n)}{n})$, then $\Delta$ with high probability does not have a complete $t'$-skeleton.

		\item For $k=1$. If $p=p_3:=1-\exp\left(\frac{-\alpha(c)}{1-c}-\frac{q \log(n)}{(1-c)n }\right)$ then  $\Delta$ has a complete $t'$-skeleton with high probability. 
		
		If instead $p=p_4:=1-\exp\left(\frac{-\alpha(c)}{1-c}+\frac{q \log(n)}{(1-c)n }\right)$, 	then $\Delta$ with high probability does not have a complete $t'$-skeleton.
	\end{enumerate}
\end{theorem}
\begin{proof}
	For $x\in [n]^{(t')}$, we write $X_x$ for the indicator of the event x is not a face. Let $X$ denote the number of sets of size $t'$ that are not faces of the complex, i.e. $X=\sum_x X_x$.

	Note that $X_x$ and $X_y$ are independent if no set of size $t$ contains $x$ and $y$, i.e. whenever $|x\cap y|< t'-k$. Thus each $A_x$ is independent from all but at most $n^{2k}$ of the $A_y$, using the crude bound that 
	there are less than $n^k$ subsets of $x$ of size $t'-k$, and each is contained in fewer than $n^k$ sets of size $t'$.

	If $X_x$ and $X_y$ are not independent, we use the easy bound, $\mathrm{Cov}(X_x, X_y)\leq \prob(X_x=1)$. Using Lemma \ref{covariance}, we have that
	\begin{align*}
		\mathrm{Var}(X)&\leq \mathbb{E}(X)+n^{2k}\binom{n}{t'} (1-p)^{\binom{n-t'}{t-t'}}. \\
		&=(n^{2k}+1)\mathbb{E}(X).
	\end{align*}
	(a) By linearity of expectation and using Proposition \ref{constcbinom},
	\begin{align*}
		\mathbb{E}(X)&= \binom{n}{cn-k} (1-p)^{\binom{n-t'}{t-t'}}\\
		&= \exp\left(\alpha(c) n+O(\log n) -p(1+o(1))\frac{(n-cn)^k}{k!}\right),
	\end{align*}
	using the approximation $(1-p)=e^{p(1+o(1))}$, valid if $p=o(1)$.

	It is easy to see that this tends to zero if $p=p_1$, and so in this case, by Markov's inequality, $\Delta$ has a complete $t'$ skeleton with high probability. 

	If instead $p=p_2$, then $\mathbb{E}(X)$ grows like a power of $n$, with the power linear in $q$,  so for $q$ large enough $\mathrm{Var}(X)=o(\expect(X)^2)$.  Hence Chebyshev's Inequality implies that 
	$X>0$ with high probability, i.e. the 	$t'$-skeleton is incomplete with high probability.

	(b) Similarly,
	\begin{align*}
		\mathbb{E}(X)&= \binom{n}{cn-1} (1-p)^{n-cn+1}\\
		&= \exp\left(\alpha(c) n+O(\log n)\right) (1-p) ^{n-cn+1}.
	\end{align*}
	If $p=p_3$, this tends to zero and hence by Markov's Inequality, $X=0$ with high probability.

	If $p=p_4$, then once more, $\mathbb{E}(X)\to \infty$ grows like a power of $n$, with the  power linear in $q$,  so for a large enough $q$  $\mathrm{Var}(X)=o(\expect(X)^2)$. Thus the $t'$-skeleton is, with 
	high probability, not complete.
\end{proof}
For  $\mathcal{RP}\left(n, \frac{n}{2}, \frac{1}{2}\right)$ this  gives a complete $(n/2-2)$-skeleton with high probability,  showing that for this property  $\mathcal{RP}\left(n, \frac{n}{2}, \frac{1}{2}\right)$  provides a correct intuition for the behaviour of $\mathcal{U}(n)$, when $n$ is even.

The case of constant $t$ is even simpler.   
\begin{theorem}\label{skeletonhypergraph}
	Let $\Delta \sim \mathcal{RP}(n,t,p)$, where $t$ is a constant and $t'<t$. If  $p=(1+\omega(n))\frac{ t'(t-t')!\log n}{n^{t-t'}}$, where $\omega(n)\log(n)\rightarrow\infty$, then with high probability, the $t'$-skeleton is complete.
	For $\omega(n)<-2(t/t'-1)$ the $t'$-skeleton is with high probability not complete.
\end{theorem}
\begin{proof}
	As before, let $X$ denote the number of $t'$-sets that are not faces of $\Delta$, and let $p$ be as in the theorem's hypothesis.
	\begin{align*}
		\mathbb{E}(X)&=\binom{n}{t'}(1-p)^{\binom{n-t'}{t-t'}}\\
		&\leq \exp\left({t'\log n-p \binom{n-t'}{t-t'}}\right)\\
		&\to 0
	\end{align*}
	and hence the $t'$-skeleton is with high probability complete.
	
	For  negative $\omega(n)=-w(n)$ we find that $\mathbb{E}(X)=n^{t' w(n) }$, and using the simple bound  for the  variance of $X$ from the previous proof we find that in order to have $\mathrm{Var}(X)=o(\expect(X)^2)$ we require 
	that $n^{2(t-t')}<n^{t' w(n) }$, or equivalently $\omega(n)< -2(\frac{t}{t'}-1)$

\end{proof}
Taking  $t'=t-1$, for a constant $t$  this shows that when $p\geq (1+\omega(n))\frac{ (t-1)\log n}{n}$, where $\omega(n)\log n\to \infty$,  the random pure model  and the Linial-Meshulam model \cite{linialmeshulam}
will be contiguous. As we have already mentioned  most of the properties which have been studied for that model have thresholds of the form $p=\frac{c}{n}$,  for probabilities of that form the two models are not equivalent.

It is well known that simplicial complex of dimension $d$ can be embedded in $\mathbb{R} ^{2d+1}$, as proven by Wegner \cite{weg}.  Is is also known, e.g. by \cite{vK33}, that if a complex on $n$ vertices has a complete 
$t$-skeleton and $t+3\leq n\leq 2t+3$ then that complex cannot be embedded in  $\mathbb{R} ^{n-3}$. Together this means that  there is a short range of dimensions into which the random uniform complex is embeddable. 
\begin{question}
	For which $d$ can the random uniform simplicial complex on $n$ vertices be embedded in $\mathbb{R} ^{d}$ with high probability?
\end{question}
We may of course as the same question for the random pure model.

\section{Homology}\label{sec:hom}
We will now turn our attention to the topological properties of random complexes from our two models.  We will first bound the rank of the homology group $H_k(\Delta,\mathbb{F})$.

\begin{proposition}
	Let $\Delta\sim \mathcal{U}(n)$, for $n$ even. With high probability, $H_k(\Delta,R)=0$,  unless $n/2-2 \leq k \leq n/2$.
\end{proposition}
\begin{proof}
	For $k> n/2 +1$ this is trivial as $\Delta$ has no face of size $n/2+2$ or higher, with high probability, and thus $C_{k}(\Delta,R)$ is empty. Similarly, for $k=n/2+1$, the result follows follows from the fact that with 
	high probability the $n/2+1$ sized faces are in bundles of one or two elements only, and thus $Z_{k}(\Delta,R)$ is empty. 

	If $k\leq n/2-3$, this follows from the fact that  with high probability the $(k+1)$-skeleton of $\Delta$ is complete.
\end{proof}

\begin{theorem}\label{uniformhomology}
	Let $\Delta\sim \mathcal{U}(n)$, for even $n$, then with high probability, $H_{n/2-1}(\Delta, R)$ has rank at least $(1+o(1))\frac{2^{(n-1)/2}}{\sqrt{\pi n}}$.
\end{theorem}
\begin{proof}[Proof of Theorem \ref{uniformhomology}]
    Let $\Delta\sim \mathcal{U}(n,n/2,A,B)$. We will show that for all  fixed admissible pairs,  $(A, B)$, the homology group $H_{n/2-1}(\Delta, \mathbb{F})$ has dimension at least $(1+o(1))2^{(n-1)/2}/\sqrt{\pi n}$ with high 
    probability. By Theorem \ref{korshunov}, this is enough to prove the result.
    
    We let $G$ be the collection of $(n/2+1)$-sets, whose subsets of size $n/2$ are all free. In other words, $G=\{x\in [n]^{(n/2+1)}: x\setminus i\in F \text{ for all } i\in x\}$. Clearly, 
    $|G|=(1+o(1))\binom{n}{n/2+1}=(1+o(1))2^{n+1/2}/\sqrt{\pi n}$. By definition of $F$, none of the sets in $G$ are faces of $\Delta$.

    We say $x\in G$ is a \emph{hole} if all subsets of $x$ are present in $\Delta$ and let $X$ denote the number of holes. We will use Chebyshev's inequality to find a lower bound on $X$ and then use this to prove a lower 
    bound on $\dim (H_{n/2-1})$.
    
    Let $X_x$ be the indicator random variable that is 1 if $x$ is a hole and 0 otherwise. Easily, $X=\sum_{x\in G} X_x$. It is easy to see that $\mathbb{E}(X_x)= \left(\frac{1}{2}\right)^{n/2+1}$ and that 
    $\mathbb{E}(X)=\binom{n}{n/2+1} 2^{-n/2-1}=(1+o(1))2^{(n-1)/2}/\sqrt{\pi n}$.
    
    If $x$, $y\in G$ have intersection of size at most $n/2-1$, then  $X_x$ and $X_y$ are independent, and hence $\mathrm{Cov}(X_x,X_y)=0$.
    
    If instead $|x\cap y|=n/2$, then $x$ and $y$ share a face of size $\frac{n}{2}$ and so
    \begin{align*}
    	\mathrm{Cov}(X_x, X_y)&=\mathbb{P}(X_x=1, X_y=1)-\mathbb{P}(X_x=1)\mathbb{P}(X_y=1)\\
	    &=2^{-n-1}-2^{-n-2}\\
	    &=2^{-n-2}
    \end{align*}

    Thus by Lemma \ref{covariance}, $\mathrm{Var}(X)\leq \mathbb{E}(X)+ n^2 |G| 2^{-n-2}=o(\mathbb{E}(X)^2)$ and so by Chebyshev's inequality, $X\geq (1+o(1))2^{(n-1)/2}/\sqrt{\pi n}$ with high probability. 
    
    A hole $x$ can be identified naturally with  a chain $v_x$, namely $v_x:=\sum_{i\in x}(x\setminus \{i\})$. Each chain $v_x$ is an element of $H_{\frac{n}{2}-1}$, by the definition of a hole.
    
    We now show that almost all of these $v_x$ are linearly independent, in fact, we prove the stronger statement that almost all the holes do not share facets with any other hole. We let $Y_x$ be the indicator random variable 
    that is 1 if $x$ is a hole and and $x$ shares a face with another hole. In other words, $Y_x=1$ if there exists an $x'$ such that $|x\cap x'|=n/2$ and both $x$ and $x'$ are holes. We let $Y=\sum_x Y_x$. Easily, 
    $\mathbb{P}(Y_x=1)\leq n^2 2^{-n-1}$ and so by linearity of expectation, $\mathbb{E}(Y)\leq n^2\binom{n}{n/2+1} 2^{-n-1} $.  By Markov's Inequality,
    \begin{align*}
    	\prob\left(Y\geq \frac{1}{n}\expect(X)\right)&\leq \frac{n\expect(Y)}{\expect(X)}\\
	    &\leq \frac{n^3\binom{n}{n/2+1} 2^{-n-1} }{\binom{n}{n/2+1} 2^{-n/2-1}}\\
	    &=n^3 2^{-\frac{n}{2}}.
    \end{align*}
    Hence,  $Y\leq \frac{1}{n}X$ with high probability, and likewise  $\dim(H_d)\geq X-Y=(1+o(1))X$. 
\end{proof}

We may use the same approach to analyse the top homology of $\mathcal{U}(n,t, A, B)$, when $n$ is odd and $\trange$.

\begin{observation}\label{oddhomology}
	If $\Delta\sim \mathcal{U}(n)$,  $n$ is odd and $\mathbb{F}$ is a field, then with high probability, either $H_{\lfloor n/2 \rfloor-1}(\Delta, \mathbb{F})$ has rank at least 
	$(1+o(1))2^{\lceil n/2 \rceil-1/2}$ or $H_{\lceil n/2 \rceil-1}(\Delta, \mathbb{F})$ has rank at least $(1+o(1))2^{\lfloor n/2 \rfloor-1/2}$.
\end{observation}

We believe that Theorem \ref{uniformhomology} is essentially best possible:
\begin{conjecture}
	Let $\mathbb{F}$ be a field and let $\Delta\sim \mathcal{U}(n)$, for even $n$. Then with high probability, $\dim(H_{n/2-1}(\Delta, \mathbb{F}))=(1+o(1))2^{(n-1)/2}/\sqrt{\pi n}$.
\end{conjecture}

The same methods allow us to prove a similar result for the random pure model. We also see that because $\alpha(1/2)=\log(n)$, the following result agrees with our heuristic that the $\mathcal{RP}(n,n/2,1/2)$ behaves similarly to $\mathcal{U}(n)$ when $n$ is even.

\begin{theorem}
	Let $\mathbb{F}$ be a field,  $\epsilon>0$ a positive constant,  and let $\Delta \sim \mathcal{RP}(n,t,p)$.
	
	(a) Let $t=cn$, for some constant $0<c<1$.  
	
	If $ p\geq e^{-\alpha(c)/c+\epsilon}$, then with high probability $H_{t}$ is non-trivial. If also $p\leq 1-\epsilon$, then 
	$\dim(H_t(\Delta,\mathbb{F}))\geq e^{(1+o(1))\alpha(c)n}p^{cn}$, with high probability.
	
	(b) Let $t$ be a constant. If $np\to \infty$, then with high probability, $H_t$ is non-trivial. If also $p\leq 1-\epsilon$, then $\dim(H_t(\Delta,\mathbb{F}))\geq (1+o(1))\frac{(np)^{t+1}}{t!}$, with high probability.
\end{theorem}

\begin{proof}
For a set $x$ of size $t+1$ , we let $X_x$ be the indicator random variable that is 1 when $x$ is a hole, that is all its subsets of size $t$ are faces of $\Delta$. We write $X=\sum X_x$. As in the proof of Theorem \ref{uniformhomology},  for both (a) and (b), we first bound $X$ using Chebyshev's inequality and then show $\dim(H_{t})$ follows the same bounds.

(a) Using Proposition \ref{constcbinom},

\[\expect(X)=\binom{n}{t+1}p^{t+1}=\exp\Big(\alpha(c)n+O(\log n)\Big)  p^{cn}.   \]

Note that for $p$ in this range, this tends to infinity. We can see also that $X_x$ and $X_y$ are independent if $|x\cap y|\leq t-1$. On the other hand, if $|x\cap y|= t$, then $Cov(X_x,X_y)\leq p^{2t+1}$. Thus $Var(X)\leq \mathbb{E}(X)+ n^2 \binom{n}{t+1} p^{2t+1}=o(\mathbb{E}(X)^2)$, and so by Chebyshev's inequality, with high probability, $X\geq \exp\Big(\alpha(c)n+O(\log n)\Big)  p^{cn}$. 

We now show that almost all of these holes are linearly independent, in fact, we show that almost all the holes do not share facets with any other hole. We let $Y_x$ be the indicator random variable that is 1 if $x$ is a hole and and $x$ shares a face with another hole. It is easy to see that $\mathbb{P}(Y_x=1)\leq n^2 p^{2t+1}$ and so by linearity of expectation, $\mathbb{E}(Y)\leq n^2\binom{n}{cn+1} p^{2cn+1} $. It is easy to use Chebyshev's inequality to show that with high probability, $Y\leq 2\mathbb{E}(Y)=o(X)$, using $p\leq 1-\epsilon$. Since $\dim(H_t)\geq X-Y$, this concludes this part of the proof.

(b) In this case,

\[\expect(X)=\binom{n}{t+1}p^{t+1}=(1+o(1))\frac{(np)^{t+1}}{t!} ,   \]
which tends to infinity since $np$ does.

Once again, $X_x$ and $X_y$ are independent if $|x\cap y|\leq t-2$, while if $|x\cap y|= t-1$, $Cov(X_x,X_y)\leq p^{2t+1}$. Thus, again, $Var(X)=o(\expect(X)^2)$ and thus by Chebyshev's inequality,  $X=(1+o(1))\frac{(np)^{t+1}}{t!}$ with high probability. Once more, we let $Y$ be the number of sets $x$ that are holes but share no facet with another hole. We may show that for $p\leq 1-\epsilon$, with high probability, $Y=o(X)$, concluding the proof.
\end{proof}

Once again, we believe that these bounds on $p$ are essentially tight. 

\begin{conjecture}
	Let $\mathbb{F}$ be a field and let $\Delta \sim \mathcal{RP}(n,t,p)$. We also let $\epsilon$ be any positive constant.\\
	(a) Let $t=cn$, for some constant $0<c<1$ .   If $ p\geq e^{-\alpha(c)/c-\epsilon}$, then with high probability $H_{t}$ is trivial. \\
	(b) Let $t$ be a constant. If $np\to 0$, then with high probability, $H_t$ is non-trivial. More precisely, $\dim(H_t(\Delta,\mathbb{F}))\geq (1+o(1))\frac{(np)^{t+1}}{t!}$, with high probability.
\end{conjecture}

\section{The Euler Characteristic}\label{sec:Euler}

In this section, we first give a concentration bound for  the Euler characteristic of $\mathcal{U}(n)$,  and also show that the Euler characteristic of $\mathcal{U}(n)$ is not too tightly concentrated.  The first results shows that the  value of the Euler characteristic is concentrated inside an interval of exponential length in $n$, and the second result demonstrates that the length of this interval cannot be shortened by more than a power of $n$. 

\begin{theorem}\label{eulerthm}
	Let $\Delta\sim \mathcal{U}(n)$, for even $n$. There is a constant, $c$, such that with high probability, $|\chi(\Delta)|\leq cn 2^{n/2}$.
\end{theorem}
\begin{proof}
Let $\Delta\sim \mathcal{U}(n)$. Recall that, $\chi=\sum_{i=1}^{n} (-1)^{i-1} f_{i}$, where $f_{i}$ denotes the number of faces of size $i$. By Proposition \ref{skeletonuniform}, we may assume that $\Delta$ has a complete $(n/2-2)$-skeleton, and thus $f_i=\binom{n}{i}$, for $i\leq n/2-2$. By Theorem \ref{korshunov}, we may assume that $\Delta\in\mathcal{M}(n,n/2,A,B)$, for some admissible $A$ and $B$. Thus we also have that $f_i=0$ for $i\geq n/2+2$ and that $f_{n/2+1}=|B|\leq 2^{n/2}$. It remains only to examine $f_{n/2}$ and $f_{n/2-1}$.

Let $F$ denote the collection of free sets. Recall that since $|A|, |B|\leq 2^{n/2}$, we have that $\binom{n}{n/2}-(n/2+1)2^{n/2+1}\leq |F|\leq \binom{n}{n/2}$. Let $X$ denote the number of free sets that  are faces of $\Delta$. This is a binomial random variable with parameters $|F|$ and $\frac{1}{2}$. It thus has mean $\frac{1}{2}|F|$ and so by Hoeffding's inequality,

\[\mathbb{P}\left(\big|X-\frac{1}{2}|F|\big|\geq 2^{n/2}\right)\leq \exp\left(-\frac{2^{n+1}}{|F|}\right)\to 0,\] since $|F|\leq \binom{n}{n/2}= O\left(\frac{2^{n}}{\sqrt{n}}\right)$.

Combined with our bounds on $|F|$, it is easy to see that $|f_{n/2}-\frac{1}{2}\binom{n}{n/2}|\leq n2^{n/2+1}$ with high probability.

Let $Y$ denote the number of $(n/2-1)$-sets that are not faces of $\Delta$. By considering $\Delta^c$, as in the proof of Proposition \ref{skeletonuniform}, we may assume that $Y\leq 2^{n/2}$ and thus that $f_{n/2-1}\geq \binom{n}{n/2-1}-2^{n/2}$. 

Thus, with high probability,

\begin{align*}
\chi&=\sum_{i=1}^{n/2-1} (-1)^{i-1}\binom{n}{i}+(-1)^{n/2-1} \frac{1}{2}\binom{n}{n/2} + O(n2^{n/2}).
\end{align*}

Note that $\sum_{i=1}^{n/2-1} (-1)^{i-1}\binom{n}{i}+\frac{1}{2}(-1)^{n/2-1}\binom{n}{n/2}=\frac{1}{2}\sum_{i=0}^{n} (-1)^{i-1}\binom{n}{i}-2$, by symmetry of $\binom{n}{i}$ around $i=n/2$. As there are equally many odd sized subsets of $[n]$ as even sized subsets, this sum evaluates to $-2$, concluding the proof.

\end{proof}

The odd case is only slightly more complicated. Suppose that $\Delta\sim \mathcal{U}(n,t)$ for $t\in \trange$. Then it is easy to use the same method as above to prove that with high probability, \[\chi(\Delta)=\sum_{i=1}^{t-1}(-1)^{i-1}\binom{n}{i}+\frac{1}{2}(-1)^{t-1}\binom{n}{t}+O(n2^{n/2}).\] Using the identity $\sum_{i=1}^{j-1} (-1)^{i-1}\binom{n}{i}=(-1)^j\frac{j}{n}\binom{n}{j}+1$, which may be verified by induction on $j$, we get in both cases that with high probability, \[ \chi(\Delta)=(-1)^{\ceil{n/2}}\frac{1}{2n}\binom{n}{\ceil{n/2}}+O(n2^{n/2}).\]

Next we will  prove an anti-concentration result for $\mathcal{U}(n)$.

\begin{theorem}\label{anticorrel}
Let $\Delta\sim \mathcal{U}(n)$, for even $n$. For any natural $k$, $\mathbb{P}(\chi(\Delta)=k)\leq cn^{1/4}2^{-n/2}$, for some constant $c$.
\end{theorem}

\begin{proof}
We fix admissible $(A,B)$ and let $\Delta\sim \mathcal{U}(n,n/2,A,B)$.

We again generate $\Delta$ by choosing free sets to be in $\Delta$ with probability $\frac{1}{2}$. However, we do this in three stages here. We will do this in such a way that the sets in the third stage do not add any $n/2-1$ sized faces to $\Delta$.

In the first stage, we choose $F_1\subseteq F$, by placing free sets in it with probability $p$, to be chosen later. We then choose each of these sets to be in $\Delta$ with probability $\frac{1}{2}$. At this point, almost all of the $(n/2-1)$-sets are in $\Delta$.

We define $F_3\subseteq F\setminus F_1$ to be the collection of free sets whose subsets of size $n/2-1$ are all already faces of $\Delta$. If $x$ is a free set not in $F_1$, then $\mathbb{P}(x\in F_3)= \left(1-\left(1-\frac{p}{2}\right)^{\frac{n}{2}}\right)^{\frac{n}{2}}.$ Thus:
\begin{align*}
\mathbb{E}(|F_3|)&= (1-p)|F|\left(1-\left(1-\frac{p}{2}\right)^{n/2}\right)^{n/2}.\\
\end{align*}

Choose $p=\frac{3\log n}{n}$. Then $\mathbb{E}(|F_3|)=(1+o(1))\binom{n}{n/2}$, and hence $\mathbb{E}(|F_2|)=o\left(\binom{n}{n/2}\right)$. It is easy to show also that $\mathrm{Var}(|F_3|)=o(\mathbb{E}(|F_3|))$. Thus, by Chebyshev's Inequality,  $\mathbb{P}(|F_3|<\frac{1}{2} \binom{n}{n/2})\leq \binom{n}{n/2}^{-1}$. Recall that $\binom{n}{n/2}=\Theta\left(\frac{2^n}{\sqrt{n}}\right)$, and so we may assume this does not happen.

Let $F_2=F\setminus (F_1 \cup F_3)$. We choose sets in $F_2$ to be faces of $\Delta$ with probability $\frac{1}{2}$, then complete the construction of $\Delta$ by repeating this process for sets in $F_3$.

We see that the probability of the Euler characteristic being a given value is at most the probability of $|F_3|$ fair coins having a given number of heads. This is at most $1/\sqrt{|F_3|}=\frac{c\cdot n^{1/4}}{2^{n/2}}$, and so we may conclude. 
\end{proof}
The same methods also show that if $n$ is odd and $\Delta \sim \mathcal{U}(n, t)$, for $\trange$, the conclusion of the theorem still holds, i.e. for any natural $k$, $\mathbb{P}(\chi(\Delta)=k)\leq cn^{1/4}2^{-n/2}$, for some constant $c$.

We believe that if the distribution of $\chi(\Delta)$ is scaled to have variance 1  and mean 0, then the scaled distribution will asymptotically converge to a normal distribution.

An immediate corollary of Theorem \ref{anticorrel} and Theorem \ref{korshunov} is the following, which we will use in Section \ref{evasivenesssec}.
\begin{corollary}\label{anticonc}
	Let $k=k(n)$ be any integer function of $n$ and let $\Delta\sim \mathcal{U}(n)$. With high probability, $\chi(\Delta)\neq k$.
\end{corollary}
Note that this implies that with high probability $\Delta$   is note of any fixed simple topological type, e.g. a topological ball.

\section{Induced Subcomplexes}\label{sec:subcomp}
Our next goal is to determine which pure simplicial complexes appear as induced subcomplexes in $\Delta\sim \mathcal{U}(n)$ with high probbility.  The answer turns out be quite simple.

\begin{theorem}
	Let $\Delta\sim \mathcal{U}(n)$ for even $n$ and let $S$ be a pure simplicial complex of dimension $n/2-1$ with $n/2+k$ vertices.\\
	(a) If $k=1$, then $S$ occurs as an induced subcomplex of $\Delta$ and there are  at least $(1+o(1))\frac{2^{n/2-1}}{\sqrt{n}}$ copies of $S$  with high probability.\\
	(b) If $k\geq 2$, then with high probability there is no copy of $S$  in $\Delta$. 
\end{theorem}

\begin{proof}
(a) Let $\Delta\sim \mathcal{U}_{n,n/2,A,B}$, for fixed admissible $(A,B)$. 

Let $G$ be the collection of sets of size $n/2+1$ such that all their subsets of size $n/2$ are free. For $x\in G$, let $X_x$ be the indicator random variable that is 1 if the faces of $x$ form a copy of S, and let $X=\sum_{x\in G} X_x$.

Since $|G|=(1-o(1))\binom{n}{n/2}$, we have that
\[\mathbb{E}(X)\geq |G|\left(\frac{1}{2}\right)^{n/2+1}=(1-o(1))\frac{2^{n/2-1}}{\sqrt{n}}.\]

If $|x\cap y|\leq n/2-2$, then $X_x$ and $X_y$ are independent random variables. If instead $|x\cap y|= n/2-1$, it is easy to see that $\mathrm{Cov}(X_x, X_y)= 2^{-n-2}$ , and so $Var(X)=o({2^{n/2}}{\sqrt{n}})$.

By Chebyshev's inequality, with high probability, $X\geq (1-o(1))\frac{2^{n/2-1}}{\sqrt{n}}$, concluding the proof.

(b) It is sufficient to prove this for $k=2$.  Consider a set $x$ of $n/2+2$ vertices, and let $s$ be the vertex set of the complex $S$.  We will first show that with high probability, at most $n^2/9$ of the $n/2$-sets of $x$ are not free.  It is easy to see deterministically that  at most $2n+2$ of the $n/2$-sets of $x$ are  contained in some $b\in B$. We thus concentrate our attention on the $a\in A$ that lie in $x$; it is enough to show there are at most $n/10$ of these. It is easy to see that the proportion of admissible pairs $(A,B)$ for which this happens is vanishingly small.

We now consider the probability that a fixed bijection, $f:s\to x$ induces a simplicial isomorphism between $S$ and the subcomplex of $\Delta$ induced by the vertices of $x$.
Due to the above calculation, we may assume that at least $\frac{n^2}{72}$ of the subsets of $x$ of size $n/2$ are free.   Thus the required probability is at most $2^{-\frac{n^2}{72}}$.

We write $X$ for the number of copies of $S$ in $\Delta$. Therefore,
\[\mathbb{P}(X)\leq \binom{n}{n/2+2}(n/2+2)!\, 2^{-\frac{n^2}{72}}.\]

This tends rapidly to zero, and so with high probability, there is no copy of $S$ in $\Delta$.

\end{proof}

It is easy to see that the same method can be used to prove similar results for odd $n$. Indeed let $\trange$ and let $S$ be a pure complex of dimension $t$ on $t+k$ vertices. Suppose that $\Delta \sim \mathcal{U}(n,t)$. If $k=1$, then $\Delta$ contains at least $(1+o(1))\frac{2^{n-t-1}}{\sqrt{n}}$ copies of $S$ with high probability. If $k=2$ then with high probability, $\Delta$ contains no copies of $S$.

\section{Shellability and the $h$-Vector}\label{shellability}\label{sec:shell}

In this section we investigate shellability and the related notions of the $h$-vector and of CM-complexes

We say that a complex is \emph{shellable} if its facets can be arranged in an order, $x_1, \dots, x_m$ in such a way that for all $i<j$,  if $x_i\cap x_j$ is non-empty, then there is some $k\leq j$ for which $x_i\cap x_j\subseteq x_k\cap x_j$ and  $|x_k\cap x_j|=|x_j|-1$. Such an order is called a \emph{shelling}. We refer to \cite{bjornerwachs} for a broader introduction ,where among other things, it is shown that if $\Delta$ is shellable, then it has a shelling with facets in non-increasing order of size.

It is easy to see that with high probability the uniform random complex is not shellable.
\begin{theorem}
	For all $n$, if $\Delta\sim \mathcal{U}(n)$ then with high probability, $\Delta$ is not shellable.
\end{theorem}
\begin{proof}
	By Theorem \ref{korshunov} it is sufficient to prove this for $\Delta\sim\mathcal{U}(n,t)$ for $\trange$. It is easy to see that having two bundles in the $(t+1)^{th}$ layer implies $\Delta$ is not shellable. 
\end{proof}

The situation is more complicated for the random pure model, and our proof does not show that  $\mathcal{RP}(n,n/2,1/2)$, which would be the direct analogue of the result for the uniform model, but comes arbitrarily close to $p=1/2$. 
\begin{theorem}\label{pure shellable}
	Let $\Delta\sim \mathcal{RP}(n,t,p)$, and let $k>0$ and $\epsilon>0$ be constants.\\
	(a) Let $t=cn$, for some constant $c$. If  $n^{-k}<p<1-\exp\left(-\frac{\alpha(c)}{2(1-c)}+\epsilon \right)$, then with high probability, $\Delta$ is not shellable.\\
	(b) Let $t>2$ be a constant. If $n^{-t+\epsilon}<p<(\frac{t-1}{2}+\frac{1}{\log(n)})\frac{\log n}{n}$ then with high probability, the complex $\Delta$ is not shellable.
\end{theorem}
\begin{proof}
Given $\Delta$ and a $(t-2)$-set, $x$, we write $G_x$ for the graph whose vertex set is the set of facets of $\Delta$ that contain $x$, and an edge is present between two facets if their intersection is a set of size $t-1$. We call this the \emph{$x$-intersection graph} of $\Delta$.

Note that if $G_x$ is not connected, then $\Delta$ is not shellable (a graph with no vertices is here considered to be connected). This is because in a shelling, after the first facet that contains $x$, each subsequent facet that contains $x$ must have an intersection of size $t-1$ with a previously added facet that contains $x$. In particular, if $G_x$ has an isolated vertex, then $\Delta$ is not shellable, unless $G_x$ is a single vertex.

Let $N_{x,v}$ be the indicator random variable that is 1 if and only if the vertex $v$ in the graph $G_x$ is isolated.

We write $N_x$ for the number of isolated vertices in the graph $G_x$, and $N$ for the total number of isolated vertices in all of these graphs. That is, $N_x=\sum_v N_{x,v}$ and $N=\sum_x N_x$. It is easy to see that:

\[
\mathbb{E}(N_x)=\binom{n-t+2}{2}p (1-p)^{2(n-t)}.
\]

Since $N_{x,v}$ depends only on the presence of $t$-sets $u$ for which $|u\cap v|\geq t-1$, the random variables $N_{x,v}$ and $N_{x',v'}$ are independent unless $|v\cap v'| \geq t-2$.

(a) Using Proposition \ref{constcbinom}, we see that:
 
\begin{align*}
\mathbb{E}(N)=& \binom{n}{cn-2}\binom{n-cn+2}{2}p (1-p)^{2(1-c)n}\\ 
&=\exp\big(\alpha(c)n +O(\log n)\big) p(1-p)^{2(1-c)n}
\end{align*}

Note that this is unimodal as a function of $p$, and hence its lowest value in the range of $p$ in the theorem is attained by one of the endpoints of the range. Thus for $p$ in the range of the theorem, this tends to infinity exponentially fast. Using the bound $\mathrm{Cov}(N_{x,v}, N_{x',v'})\leq\prob(N_{x,v}=1)= p(1-p)^{2(n-t)}$
\begin{align*}
\mathrm{Var}(N)&\leq \mathbb{E}(N) + n^4 \binom{n-cn+2}{2} \binom{n}{cn-2} p (1-p)^{2(1-c)n}\\
&= (1+n^4)\mathbb{E}(N)\\
&=o(\expect(N)^2).
\end{align*}
Thus Chebyshev's Inequality implies that $N=(1+o(1))\mathbb{E}(N)$  with high probability.

Let $M_x$ denote the indicator random variable that is 1 if $G_x$ consists of only one vertex, and let $M=\sum_x M_x$.

\[\mathbb{E}(M)= \binom{n}{cn-2}\binom{n-cn+2}{2}p (1-p)^{\binom{n-cn+2}{2}-1}.\]

Since $p>n^{-2+\epsilon}$, this is a smaller order of magnitude than $\expect(N)$, so by Markov's Inequality, $M=o(N)$ with high probability. Thus with high probability, $N-M>0$ and so the complex is not shellable.

(b) 
We again use the linearity of expectation to see that:
\begin{align*}
\expect(N)&=(1+o(1))\frac{n^{t-2}}{(t-2)!}\frac{n^2}{2}p(1-p)^{2(n-t)}\\
&=(1+o(1))\frac{n^{t}}{2(t-2)!}p \exp(-2pn)\\
\end{align*}

For $p$ in the range of the theorem, this tends to infinity. This time we require a slightly more delicate argument to bound the covariances. We make use of the fact that $\mathrm{Cov}(N_{x,v}, N_{x',v'})\leq \prob(N_{x,v}=1 \text{ and } N_{x',v'}=1)\leq (1-p)^{4(n-t)-f(x,x',v,v')}$, where $f(x,x',v,v')$ denotes the number of $t$-sets, $w$, such that $x, x'\subseteq w$, $|v\cap w|=t-1$ and $|v'\cap w|=t-1$.       If $|v\cap v'|=t-2$, it is easy to see that $f(x,x',v,v')\leq 4$. If $|v\cap v'|=t-1$, we can see that   $f(x,x',v,v')\leq n$; again, we do not require the condition that $x, x'\subseteq w$ for this bound.   Suppose instead that $v=v'$. We assume $x\neq x'$ (otherwise the covariance is zero). Then $|x\cup x'|\geq t-1$, so we have at most $n-t$ choices for $w$, and so $f(x,x',v,v')\leq n-t$.   Hence,
\begin{align*}
\mathrm{Var}(N)\leq &\mathbb{E}(N)+t^2n^2\binom{n-t+2}{2}\binom{n}{t-2} p^2(1-p)^{4n-4t-4}\\
&\phantom{\mathbb{E}(N)}+tn\binom{n-t+2}{2}\binom{n}{t-2}p^2(1-p)^{3n-4t}\\
=&o(\mathbb{E}(N)^2).
\end{align*}

Thus by Chebyshev's Inequality, $N=(1+o(1))\mathbb{E}(N)$  with high probability.

Let $M_x$ denote the indicator random variable that is 1 if $G_x$ consists of only one vertex, and let $M=\sum_x M_x$.

\[\mathbb{E}(M_x)= \binom{n}{t-2}\binom{n-t+2}{2}p (1-p)^{\binom{n-t+2}{2}-1}.\]

Once again, this is a smaller order of magnitude than $\expect(N)$, so by Markov's Inequality, $M=o(N)$ with high probability. Thus with high probability, $N-M>0$ and so the complex is not shellable.

\end{proof}
The upper bound for constant $t$ could be improved to $\frac{c\log(\log(n))}{\log(n)}$, at the cost of more extensive calculations.

Next we will introduce an important concept related to the well studied class of Cohen-Macaulay complexes, or CM-complexes for short.  A simplicial complex is said to be Cohen-Macaulay if $\bar{H}_i( Lk( x  )   ,R)=0$  for $i< dim(Lk(x))$ for every $x\in \Delta$. The original definition given by Stanley \cite{Stan75}  was stated in terms of a polynomial ring associated to $\Delta$   and the characterisation used here was given by Reisner \cite{Rei76}.

For a pure complex $\Delta$, with facets of size $t$,   the \emph{$h$-vector} $$h(\Delta)=(h_0(\Delta),h_1(\Delta), \dots h_t(\Delta))$$ is defined by the relation $h_k=\sum_{i=0}^k (-1)^{k-i} \binom{t-i}{k-i} f_{i}$, where, as before, $f_{i}$ denotes the number of faces of size $i$. In particular, $h_t=(-1)^t\sum_{i=0}^t(-1)^i f_i$, and is a multiple of the Euler characteristic.

In the range of large $t$, we are able to find ranges of $p$ for which $h_t(\Delta)<0$ with high probability, if $\Delta\sim \mathcal{RP}(n,t,p)$. It is well known that all CM-complexes have non-negative $h$-vector, so this gives a lower bound on the threshold for being a CM complex

\begin{theorem}
Let  $\Delta\sim \mathcal{RP}(n,t,p)$, let $\epsilon$ be a positive constant and let $t=cn$ for some constant $0<c<1$. If $ \frac{2\alpha(c)+\epsilon}{2(1-c)^2 n }<p <\frac{t-\epsilon}{n}$, then with high probability, $h_t<0$, and for $p >\frac{t-\epsilon}{n}$, we have $h_t>0$ with high probability.

 (b) Let $t>2$ be a constant.  If $(1+\omega(n))\frac{2(t-2)\log n}{n^2}<p< \frac{t-\epsilon}{n}$, then with high probability, $h_t<0$,  and for $p >\frac{t-\epsilon}{n}$, we have $h_t>0$ with high probability.  
\end{theorem}

\begin{proof}
We will treat both cases together and then use the lower bounds on $p$ to conclude using Propositions \ref{skeletonpure} and \ref{skeletonhypergraph}.

Firstly, we note that $f_{t}$ is binomially distributed, with parameters $\binom{n}{t}$ and $p$. Thus $\expect(f_{t})=p\binom{n}{t}$ and $\mathrm{Var}(f_{t})=p(1-p)\binom{n}{t}=o(\expect(f_{t})^2)$, so Chebyshev's Inequality allows us to conclude that with high probability, $f_{t}=(1+o(1))p\binom{n}{t}$.

For sets $x\in [n]^{(t-1)}$, we write $X_x$ for the indicator random variable that is 1 if $x$ is not a face of $\Delta$ and let $X=\sum_x X_x$. Clearly, $\expect(X)=\binom{n}{t-1}(1-p)^{n-t+1}$. Since $f_{t-1}=\binom{n}{t-1}-X$, we may use Markov's Inequality to show that with high probability, $f_{t-1}=(1+o(1))\binom{n}{t-1}$.

By  Propositions \ref{skeletonpure} and \ref{skeletonhypergraph}, with high probability, $f_{i}=\binom{n}{i}$ for all $i\leq t-2$. Thus with high probability,

\[h_{t}=\sum_{i=0}^{t-1} (-1)^{t-i} \binom{n+1}{i} +(1+o(1))p\binom{n}{t} +o\left(\binom{n}{t-1}\right). \]

Since $\sum_{i=0}^j (-1)^{j-i}\binom{n}{i}=-\frac{j+1}{n}\binom{n}{j+1}$, as seen by induction on $j$, the above simplifies to $h_t=  ((1+o(1))p- \frac{t}{n})\binom{n}{t}+ o\left(\binom{n}{t-1}\right) $, which concludes the proof.

\end{proof}
Since the $h$-vector is positive for $p >\frac{t-\epsilon}{n}$ it is natural to ask for which range of $p$  the random uniform complex is also Cohen-Macaulay.

\section{Evasiveness}\label{evasivenesssec}
In this section  we will use our results about $\mathcal{U}(n)$ to prove a result about the complexity of Boolean functions, thus reconnecting with Korshunov's original formulation of his theorem.

A \emph{Boolean function} is a function $f:\{0,1\}^n\to \{0,1\}$. Informally, the \emph{decision tree complexity of $f$}, written $D(f)$, is the minimum number of bits of input an adaptive algorithm must know in order to determine the output. More precisely, the algorithm initially has no information about the input, $x=(x_1, \dots, x_n)$, and at each stage asks a question of the form `what is $x_i$?', where its choice of question may depend on previously received answers. The algorithm terminates when it knows the value of $f(x)$. We say $D(f)\leq t$, if there is some algorithm that terminates after at most $t$ questions for all inputs $x$. Trivially, $D(f)\leq n$. We say $f$ is \emph{evasive} if $D(f)=n$. Trivially, the projection functions,  $p_i(x_1,\dots, x_n)=x_i$, are not evasive (for $n>1$), while the layer functions, $l_i(x_1,\dots, x_n)=\sum x_i$, are evasive.

In 1976, Rivest and Vuillemin \cite{rivestvuillemin} showed that almost all Boolean functions are evasive, in the sense that the proportion of functions on $n$ variables that are evasive tends to 1, as $n\to \infty$. 

A symmetry of a Boolean function $f$ is a permutation $\phi:[n]\to [n]$ such that $f(x_1,\dots, x_n)=f(x_{\phi(1)}, \dots x_{\phi(n)})$, for all $(x_1,\dots, x_n)$. We say $f$ is \emph{transitive} if it has a transitive group of symmetries, that is, for all $i$ and $j$, there is a symmetry $\phi$ such that $\phi(i)=j$.

A Boolean function, $f$, is \emph{monotone} if $f(x_1,\dots,x_n)\leq f(y_1,\dots,y_n),$ whenever $x_i\leq y_i$ for all $i$. For $A\subseteq [n]$, let $x_A$ denote the Boolean string whose $i^{th}$ co-ordinate is 1 if and only if $i\in A$. We write $\Delta_f$ for the subset of $\mathcal{P}[n]$ that consists of all non-empty subsets $A$ such that $f(x_A)=0$. If $f$ is a monotone function, $\Delta_f$ is a simplicial complex. Note that every simplicial complex may be associated with a monotone Boolean function in this way. The usefulness of this viewpoint when studying evasiveness is demonstrated by the following lemma.

\begin{lemma}[Kahn-Saks-Sturtevant \cite{kahnetal}]
	If $f$ is a non-constant monotone Boolean function that is not evasive, then $\Delta_f$ is contractible.
\end{lemma}

Much of the work in the area of evasiveness is concerns the following conjecture, which is a generalised form of a conjecture by Karp.
\begin{conjecture}[Evasiveness Conjecture]
	Every non-constant monotone, transitive Boolean function is evasive.
\end{conjecture}
If the condition of transitivity is relaxed, the function need not be evasive, as demonstrated by the projection functions as shown above. However, the following easy corollary of our work on the uniformly random simplicial complex shows that this is the exception rather than the rule.
\begin{corollary}
	Almost all monotone Boolean functions are evasive.
\end{corollary}
\begin{proof}
	It is sufficient to show that almost all simplicial complexes are not contractible. Since contractible complexes have trivial homology group, this follows from Theorem \ref{uniformhomology} and Observation \ref{oddhomology}.
\end{proof}
We note that that this result can also be derived from our bounds for the Euler characteristic of $\mathcal{U}(n)$.

\section*{Acknowledgments}
This first author was supported by The Swedish Research Council grant 2014-4897.  This work was done during second author's visit at Ume\aa \ University in 2015 supported by the Eileen Colya prize.

\bibliographystyle{plain}

\end{document}